\documentclass[english,11pt]{amsart}
\usepackage{amsopn,amsthm,amsfonts,amsmath,amssymb,amscd}
\usepackage{babel}
\usepackage{amssymb,tikz}

\newtheorem{Theorem}{Theorem}[section]
\newtheorem{Lemma}[Theorem]{Lemma}
\newtheorem{Proposition}[Theorem]{Proposition}

\newtheorem*{Def}{Definition}

\newtheorem*{Conjecture}{Conjecture}

\newenvironment{Proof*}{{\it Proof.}}

\newcommand{\ZZ}{\mathbb{Z}}

\newcommand{\CCC}{\mathcal{C}}
\newcommand{\AAA}{\mathcal{A}}
\newcommand{\II}{\mathcal{I}}
\newcommand{\Dd}{\mathcal{D}}

\newcommand{\DEF}[1]{\emph{#1}}

\newcommand{\ch}{{\rm char}}

\newcommand{\diam}[1]{{\rm diam}(#1)}

\begin{document}

\title{The total graphs of finite rings}
\author{David Dol\v zan, Polona Oblak}
\date{\today}

\address{D.~Dol\v zan:~Department of Mathematics, Faculty of Mathematics
and Physics, University of Ljubljana, Jadranska 19, SI-1000 Ljubljana, Slovenia; e-mail: 
david.dolzan@fmf.uni-lj.si}
\address{P.~Oblak: Faculty of Computer and Information Science, University of Ljubljana,
Tr\v za\v ska 25, SI-1000 Ljubljana, Slovenia; e-mail: polona.oblak@fri.uni-lj.si}

\bigskip

\begin{abstract} 
 In this paper we extend the study of total graphs $\tau(R)$ to non-commutative finite rings $R$. We prove that $\tau(R)$ is connected if and only if 
 $R$ is not local and we see that in that case $\tau(R)$ is always Hamiltonian. We also find an upper bound 
 for the domination number of $\tau(R)$ for all finite rings $R$.
\end{abstract}

\maketitle 

\section{Introduction}

\bigskip

In \cite{AndBad08}, Anderson and Badawi introduced the notion of a total graph of a 
commutative ring $R$ as the graph with all elements of $R$ as vertices, and for distinct
$x, y \in R$, the vertices $x$ and $y$ are adjacent if and only if $x + y$ is a zero-divisor in $R$. 
They studied some graph theoretical parameters of this graph such as diameter and girth.
In addition, they studied some special subgraphs of the total graph, and the properties
of the total graph based on these subgraphs. They also proved that the total graph of
a commutative ring is connected if and only if the set of zero-divisors does not form an ideal.
In \cite{akbari} Akbari et~al.~proved that if the
total graph of a finite commutative ring is connected then it is also a Hamiltonian graph.
In \cite{maimani}, Maimani et~al.~gave the necessary and sufficient conditions for the total
graphs of  finite commutative rings to be planar or toroidal and in \cite{TaChAsir13} Tamizh Chelvam and Asir
characterized all commutative rings such that their total graphs have genus $2$.

In \cite{Shekarriz}, Shekarriz et~al.~studied the total graph of a finite commutative ring and calculated the domination
number of such a ring and also found the necessary and sufficient conditions for the graph to be Eulerian.

\bigskip

In this paper, we extend the study of total graphs to the setting of arbitrary (possibly non-commutative) finite rings. 
In accordance with \cite{AndBad08}, we define the total graph of a ring $R$ as follows.

\begin{Def}
 The total graph $\tau(R)$ of a  ring $R$ is the  graph, where
  \begin{itemize}
   \item the set of vertices $V(\tau(R))$ of the graph $\tau(R)$ is the set of all elements in $R$ and
   \item two distinct vertices $x$ and $y$ are adjacent if and only if $x + y$ is a (left or right) zero-divisor in $R$. 
 \end{itemize}
\end{Def}

Recall that in finite rings, every left zero-divisor is also a right zero-divisor and vice versa, so the 
definition of a total graph here coincides with the commutative version of definition in \cite{AndBad08}.
We limit our study to the total graphs of finite rings for the following lemma. 

\begin{Lemma}[\cite{ganesan,koh}]
 If $R$ is a ring with $m$ zero divisors, $2 \leq m < \infty$, then $R$ is a finite ring
 with $\vert R \vert \leq m^2$.
\end{Lemma}

So, if an infinite ring $R$ has more than one zero divisor, then it has infinitely many of them
and so the degree of each vertex in the total graph is infinite, which means that is it difficult (or perhaps
even meaningless) to study the graph theoretical properties such as being Eulerian, the domination
number, etc.  On the other hand, if an infinite ring $R$ has only one zero divisor, each $a \in R$ is either
an isolated vertex or adjacent only to $-a$, so the total graph is a disjoint union of infinitely many 
graphs isomorphic to $K_1$ or $K_2$.

\medskip

We will also often make use of the fact that the Jacobson radical of an Artinian (hence also finite) ring is nilpotent.

\bigskip

In the paper, we use the following notations.

For any ring $R$, we denote by $Z(R)$ the set of zero-divisors, $Z(R)=\{x \in R;$ there exists
$0 \neq y \in R \text { such that } xy=0 \text { or } yx=0 \}$, and by $R^*$ the set of all invertible 
elements of $R$.  By $J=J(R)$ we will denote the Jacobson ideal of the ring $R$.

We denote by $M_n(F)$ the set of all $n \times n$ matrices over a field $F$. The matrix with the only nonzero entry 1 in the $i$-th row and $j$-th column will be denoted by $E_{i,j}$ and
we will denote the zero matrix by $0$.

The sequence of edges $x_0 - x_1$, $ x_1 - x_2$, ..., $x_{k-1} - x_{k}$ in a graph is called \emph{a path of length $k$}. The \DEF{distance} between two vertices is the length of the shortest
path between them. The \DEF{diameter} $\diam{\Gamma}$ of the graph $\Gamma$ is the longest
 distance between any two vertices of the graph.
A path $x_0 - x_1 - \ldots - x_{k-1} - x_0$ is called a \emph{cycle}. 
A \emph{Hamiltonian path} of a graph $G$ is a path that contains every vertex of $G$ and a \emph{Hamiltonian cycle} of a graph $G$ is a cycle that contains every vertex of $G$. A graph is \emph{Hamiltonian}
if it contains a Hamiltonian cycle.  A graph $G$ is \emph{Eulerian}
if it contains a cycle that consists of all the edges of $G$.
A complete graph on $m$ vertices will be denoted by $K_m$, and a complete bipartite graph with the respective sets of sizes $m$ and $n$ will be denoted by $K_{m,n}$.

\bigskip

This paper is organised as follows. In the preliminary section, we recall some known results about 
total graphs on commutative rings and list some properties of  total graphs over non-commutative finite rings that can be proved by similar arguments as in the commutative case.
We also prove that the total graph of a local ring is not connected.

The methods we use in the remainder of the paper  differ  substantially from the ones used in studying the commutative case, where the ring decomposes as the product of local rings.
In Section 3, we generalize \cite[Theorem 3]{akbari} and prove that the total graph of a non-local (non-commutative) finite ring is Hamiltonian. In Section 4, we give the upper bound for a domination number of a finite ring and a conjecture about the exact value of a domination number for a certain class of finite rings.


\bigskip
\bigskip

\section{Preliminaries}

\bigskip


We shall need some well-known facts about rings: if $R$ is a semisimple Artinian ring, then
$R$ is a finite direct product of full matrix rings.  The following is also commonly known however, we include the proof for the sake of completeness.

\bigskip

\begin{Lemma}\label{lemma:finite}
 If $R$ is a finite ring then every $a\in R$  is either invertible or zero-divisor.
\end{Lemma}

\begin{proof}
  Choose any $a \in R$. Since $R$ is finite, there exist integers $k,l>0$ such
  that $a^k=a^{k+l}$. Choose smallest such $k$. Then 
  $a^k(a^l-1)=0$ and either $a^l=1$ (so $a$ is invertible) or $aa^{k-1}(a^l-1)=0$ (so
  $a$ is a zero divisor).
\end{proof}

We shall often use the properties of the factor ring over the Jacobson ideal and thus the following lemma will be useful.

\bigskip

\begin{Lemma}\label{lemma:R/J}
 For every finite ring $R$ we have
  $$a+b \in Z(R) \, \text{ if and only if } \, (a+J)+(b+J) \in Z(R/J).$$
\end{Lemma}

\begin{proof}
If $a+b \in Z(R)$, there exists a nonzero $c \in R$ such that $(a+b)c=0$, and thus 
$((a+J)+(b+J))(c+J)=J$.  So, either $(a+J)+(b+J)$ is a zero-divisor in $R/J$ or it is invertible. If $(a+J)+(b+J)$ is invertible, there exists $u \in R$ such that $(u+J)((a+J)+(b+J))=1+J$,
or equivalently $u(a+b) \in 1+J$. Since $R$ is finite, all elements in $1+J$ are invertible, so $u(a+b)$ is an invertible element in $R$, which contradicts
the assumption $a+b \in Z(R)$. Thus, $(a+J)+(b+J) \in Z(R/J)$.

If $(a+J)+(b+J) \in Z(R/J)$, there exists  $c \in R \backslash J$ such that 
$((a+J)+(b+J))(c+J)=J$ and thus $(a+b)c \in J$. If $a+b$ is invertible in $R$, there exists $u$ such that
$u(a+b)=1$ and therefore  $c= u(a+b)c \in J$, a contradiction. Thus,  $a+b$ is not invertible, so  $a+b \in Z(R)$.
\end{proof}

\bigskip

The generalisation of Theorems 2.1 and 2.2 from \cite{AndBad08} to the non-commutative rings can be easily proved by Lemma \ref{lemma:R/J}, as we show in the following proposition.

\bigskip

\begin{Proposition}\label{thm:local}
If  $R$ is a local ring, then $\tau(R)$ is not connected. 

Moreover, $\tau(R)$ is isomorphic to the union of $|R/Z(R)|$ copies of $K_{|Z(R)|}$, if $\ch(R)= 2^k$, 
and $\tau(R)$ is isomorphic to the union of $K_{|Z(R)|}$ and  $\frac{1}{2}(|R/Z(R)|-1)$ copies of 
$K_{|Z(R)|,|Z(R)|}$ otherwise.
\end{Proposition}

\begin{proof}
If  $R$ is a local ring, then $J=Z(R)$ and thus $F=R/Z(R)$ is a field. So, distinct vertices $a$ and $b$ are connected by an edge in $\tau(F)$ if and only if $b=-a$. Thus, if  $\ch(R)=2^k$, then $\ch(F)=2$ and therefore
 $\tau(F)$ consists of $|F|$ components equal to $K_1$. Otherwise, if $\ch(R)\ne 2^k$, then $\ch(F)\ne 2$ and so  $\tau(F)$ consists 
of $\frac{1}{2} (|F|+1)$ components, one of them corresponding to isolated vertex $0$, others equal to $K_2$. By  Lemma \ref{lemma:R/J},  the proposition follows.
\end{proof}

\bigskip

Some properties of the total graph over a non-commutative finite ring can be proved by the same arguments as in the commutative case. For example,  the arguments in the proofs of 
Theorems 2.7 and 3.3  from \cite{Shekarriz}
and Theorem 3.3 and Theorem 3.4 from \cite{AndBad08} are valid also in the non-commutative case, see Lemmas \ref{lem1}, \ref{lem2} and \ref{lem3}.

\bigskip

\begin{Lemma}[\cite{Shekarriz}]\label{lem1}
 Let $R$ be a finite ring.
 \begin{enumerate}
  \item[(a)] If $|R|$ is even, then $\tau(R)$ is a $(|Z(R)|-1)$-regular graph.
  \item[(b)] If $|R|$ is odd, then $\deg(a)=|Z(R)|$ if $a \in R^*$ and $\deg(a)=|Z(R)|-1$ if $a \in Z(R)$.
 \end{enumerate} 
\end{Lemma}

\bigskip

\begin{Lemma}[\cite{Shekarriz}]\label{lem2}
 If $R$ is a finite ring, $\tau(R)$ is Eulerian if and only if $R$ is isomorphic to a direct product of at least two finite fields of even orders.
\end{Lemma}

\bigskip

\begin{Lemma}[\cite{AndBad08}]\label{lem3}
 If $R$ is a finite ring, then 
  $$\diam{R} = \begin{cases}
     \infty, & \text{if }  R \text{ is  local,}\\
     2, & \text{if } R \text{ is not local.}
     \end{cases}$$
\end{Lemma}

\bigskip
\bigskip
\section{Total graph of a non-local ring is Hamiltonian}
\bigskip

The total graph of a finite local ring is disconnected by Proposition \ref{thm:local}. So, in order to study the existence of a Hamlitonian cycle, we have to limit ourselves to the non-local case.  In this section, we shall prove that the total graph of any non-local finite ring is Hamiltonian.  In the course of the proof, we will use the following notations.

First, we define the index sets 
$$\II_{k,l} =\{ (i,j); i=k  \text{ and } j<l, \text { or } i<k  \}, \quad \overline{\II_{k,l}} =\II_{k,l} \cup \{(k,l)\}$$
and
$$\underline{\II_{k,l}}= \begin{cases}
 \II_{k,l} \backslash \{(k-1,n)\}, & \text{if } l=1,\\
 \II_{k,l} \backslash \{(k,l-1)\}, & \text{if } l\ne 1. 
\end{cases}$$

Next, we define the sets of matrices corresponding to these index sets, 
$$\AAA_{k,l}=\{[a_{ij}] \in M_n(F); \; a_{i,j}=0 \text{ if } (i,j) \notin \II_{k,l} \} \subseteq V(\tau(M_n(F)))$$
and
$$\overline{\AAA_{k,l}}=\{[a_{ij}] \in M_n(F); \; a_{i,j}=0 \text{ if } (i,j) \notin \overline{\II_{k,l}} \} \subseteq V(\tau(M_n(F))).$$

\medskip

Note that  $\{0\}=\AAA_{1,1} \subseteq \AAA_{1,2} \subseteq \ldots \subseteq \AAA_{1,n} \subseteq \AAA_{2,1} \subseteq\ldots \subseteq \AAA_{n,n} \subseteq \overline{\AAA_{n,n}}=M_n(F)$.

Finally, for any $c \in F$, we shall also make use of the following sets,

$$\CCC_{k,l}(c)=\{[a_{ij}] \in M_n(F); \; a_{i,j}=c \text{ if } (i,j) \in \II_{k,l} \},$$
$$\overline{\CCC_{k,l}}(c)=\{[a_{ij}] \in  M_n(F); \; a_{i,j}=c \text{ if } (i,j) \in \overline{\II_{k,l}} \}$$
and
$$\underline{\CCC_{k,l}}(c)=\{[a_{ij}] \in  M_n(F); \; a_{i,j}=c \text{ if } (i,j) \in \underline{\II_{k,l}} \}.$$

\bigskip

The following two lemmas will be crucial in proving that the total graph of a full matrix ring over a field is Hamiltonian.

\bigskip

\begin{Lemma}\label{thm:Ham0}
Let  $\ch(F) = 2$ and $1 \leq k, l \leq n$.
 If  there exists a Hamiltonian path on $\AAA_{k,l}$ with the first vertex 0 and the last vertex in $\underline{\CCC_{k,l}}(0)$, then there exists
 a Hamiltonian path on $\overline{\AAA_{k,l}}$ with  the first vertex 0 and the last vertex in $\CCC_{k,l}(0)$.
\end{Lemma}

\begin{proof}
 Let $F=\{0,x_1,\ldots,x_m\}$.
Since $\ch(F)=2$, $m$ is odd and we can extend the Hamiltonian  path $0,B_1,\ldots, B_{t}$ on $\AAA_{k,l}$ by the path 
\begin{center}
\begin{tabular}{ccccccc}
   $x_1 E_{k,l}$ &\, \textemdash \,&$(B_1+x_1 E_{k,l})$  &\, \textemdash \,&\ldots  &\, \textemdash \, &$(B_t+x_1 E_{k,l})$ \\
  $ |$ & &  & &  & & \\
   $x_2 E_{k,l}$  &\, \textemdash \,&$(B_1+x_2 E_{k,l})$ &\, \textemdash \,&\ldots  &\, \textemdash \,&$(B_t+x_2 E_{k,l})$\\
    & &  & &  &  &$|$ \\
   $x_3 E_{k,l}$ &\, \textemdash \,&$(B_1+x_3 E_{k,l})$  &\, \textemdash \,&\ldots  &\, \textemdash \, &$(B_t+x_3 E_{k,l})$ \\
   $|$ & &  & &  & & \\
   \vdots  &&\vdots && \vdots  & &\vdots\\
   & &  & &  & & $|$\\ 
   $x_m E_{k,l}$  &\, \textemdash \,&$(B_1+x_m E_{k,l})$  &\, \textemdash \,&\ldots  &\, \textemdash \,&$(B_t+x_m E_{k,l})$
 \end{tabular}
 \end{center}
 starting in $B_t+x_1 E_{k,l}$ and ending in $x_m E_{k,l}\in \CCC_{k,l}(0)$.
\end{proof}

\bigskip

\begin{Lemma}\label{thm:Ham}
 Let  $\ch(F) \neq 2$ and $1 \leq k, l \leq n$. If there exists a Hamiltonian path on $\AAA_{k,l}$ with the first vertex 0 and the last vertex in $\CCC_{k,l}(c)$ for some
$c \in F^*$, then there exists
 a Hamiltonian path on $\overline{\AAA_{k,l}}$ with  the first vertex 0 and the last vertex in $\overline{\CCC_{k,l}}(d)$ for
 some $d \in F^*$.
\end{Lemma}

\begin{proof}
Suppose there  exists a Hamiltonian path $0,B_1,\ldots, B_{t}$ on $\AAA_{k,l}$ with the first vertex 0 and the last vertex in $B_t \in \CCC_{k,l}(c)$.
Since $|F|$ is odd, let us order the elements $F=\{0,x_1,-x_1,\ldots,x_m,-x_m\}$, where $x_m=(-1)^{m} c$. Note that $t$ is odd, so we can extend the Hamiltonian path   
by the path 

{\small
\begin{tabular}{ccccccccc}
   $x_1 E_{k,l}$ &\, \textemdash &$(-B_1- x_1 E_{k,l})$   &\, \textemdash \, &$(-B_2+x_1 E_{k,l})$   &\, \textemdash \, &\ldots  &\, \textemdash \,  &$(-B_t+x_1 E_{k,l})$ \\
   $\vert$  &  & &&&  && & \\
   $- x_1 E_{k,l}$ &\, \textemdash  &$(-B_1+x_1 E_{k,l})$ &\, \textemdash \, &$(-B_2- x_1 E_{k,l})$   &\, \textemdash \, &\ldots  &\, \textemdash \,  &$(-B_t-x_1 E_{k,l})$ \\
    & &  &   &  &&&&$\vert$ \\
   $x_2 E_{k,l}$ &\, \textemdash  &$(B_1- x_2 E_{k,l})$    &\, \textemdash \, &$(B_2+x_2 E_{k,l})$   &\, \textemdash \, &\ldots  &\, \textemdash \,  &$(B_t+x_2 E_{k,l})$ \\
   $\vert$ & & &  &&&& & \\
   $- x_2 E_{k,l}$ &\, \textemdash  &$(B_1+x_2 E_{k,l})$  &\, \textemdash \, &$(B_2- x_2 E_{k,l})$   &\, \textemdash \, &\ldots  &\, \textemdash \,  &$(B_t-x_2 E_{k,l})$ \\
    & &  & &  & & &&$\vert$ \\
   \vdots  &&\vdots && \vdots  & &\vdots\\
   $\vert$&   & && & && & \\ 
  $ -x_m E_{k,l} $ &\, \textemdash  & $((-1)^m B_1+x_m E_{k,l})$  &\, \textemdash \, &$((-1)^mB_2- x_m E_{k,l})$   &\, \textemdash \, &\ldots  &\, \textemdash \, &$((-1)^m B_t- x_m E_{k,l})$
 \end{tabular}
}
 starting in $-B_t+x_1 E_{k,l}$ and ending in $(-1)^mB_t-x_m E_{k,l}$.

Since $(-1)^m B_t- x_m E_{k,l}=(-1)^m(B_t+c E_{k,l}) \in \overline{\CCC_{k,l}}(d)$, where $d=(-1)^m c$, the Lemma follows. 
\end{proof}

\bigskip

We can now prove that the total ring of a full matrix ring is Hamiltonian.

\bigskip

\begin{Lemma}\label{thm:Ham1}
Suppose $n \geq 2$ and $F$ is a field. Then the graph $\tau(M_n(F))$ is Hamiltonian.
\end{Lemma}

\begin{proof}
Note that for any $d \in F$ the sets $\overline{\CCC_{n,n}}(d)$ and $\CCC_{n,n}(0)$ consist of matrices that are zero-divisors, and $\overline{\AAA_{k,l}}=\AAA_{k,l+1}$ if $l <n$ and
$\overline{\AAA_{k,n}}=\AAA_{k+1,1}$ if $k <n$.
Therefore the Hamiltonian path on $\overline{\AAA_{n,n}}$ constructed by inductively applying Lemma \ref{thm:Ham0} in the
case of $\ch(F)=2$, or by inductively applying Lemma \ref{thm:Ham} in the
case of $\ch(F) \neq 2$, actually gives rise to a Hamiltonian cycle. 
\end{proof}

\bigskip

To prove that the total graph of a non-local finite ring is Hamiltonian, we also need to examine the total graphs of direct products of rings and the total graphs of factor rings.

\bigskip

\begin{Lemma}\label{thm:Ham2}
 If $R$ and $S$ are finite rings, then $\tau(R\times S)$ is Hamiltonian.
\end{Lemma}

\begin{proof}
 Let $R=\{a_1,a_2,\ldots,a_r\}$ and $S=\{b_1,b_2,\ldots,b_s\}$.  Note that $(x,y)$ is a zero-divisor in $R \times S$ if and only if 
 $x \in Z(R)$ or $y \in Z(R)$. 
 
 Suppose first that $\ch (S)=2$. Then $s$ is even and thus $(a_1,b_{1}) - (a_2,b_1) - \ldots - (a_r,b_{1}) - (-a_r,b_2) - \ldots - (-a_1,b_{2}) - \ldots - (-a_1,b_s) - (a_1,b_1)$
 is a Hamiltonian cycle in   $\tau(R\times S)$. If $\ch(R)=2$, we can reverse the roles of $R$ and $S$ to obtain the same result.
 
 If $\ch(R) \ne 2$ and $\ch(S) \ne 2$, we can reorder the elements of $R$ so that $a_r=-a_1$ and the elements of $S$ so that
  $S=\{b_1,-b_1,b_2,-b_2,\ldots,b_t,-b_t, b_{t+1},\ldots,b_s\}$, where $2b_{i}=0$ for $i=t+1,\ldots,s$.
  Note that for $i=1,2,\ldots, \left\lceil \frac{t}{2}\right\rceil$ the graph $\tau(R\times S)$ contains paths 
  $(a_1,b_{2i-1}) - (a_1,-b_{2i-1}) - (a_2,b_{2i-1}) - (a_2,-b_{2i-1}) - \ldots - (a_r,b_{2i-1}) - (a_r,-b_{2i-1}) $, for  $i=1,2,\ldots, \left\lfloor \frac{t}{2}\right\rfloor$ it contains paths
  $(-a_r,b_{2i}) - (-a_r,-b_{2i})  - (-a_{r-1},b_{2i}) - (-a_{r-1},-b_{2i})  - \ldots - (-a_1,b_{2i}) - (-a_1,-b_{2i})$   and it also
contains the path 
   $(a_1,b_{t+1}) - (a_2,b_{t+1}) - \ldots - (a_r,b_{t+1}) - (-a_r,b_{t+2}) - \ldots - (-a_1,b_{t+2}) - \ldots - (a,b_s)$, where $a$ is $-a_1=a_r$.
   We can join all these paths into a Hamiltonian cycle $(a_1,b_1)$ \textemdash$(a_1,-b_1)$  \textemdash $(a_2,b_1)$ \textemdash \ldots  \textemdash $(a_r,-b_{1})$ \textemdash$(-a_r,b_{2})$ \textemdash
   \ldots  \textemdash $(a,b_{s})$  \textemdash $(a_1,b_1)$.
 \end{proof}

\bigskip

Since $R/J$ is a finite semisimple ring, it follows by the Wedderburn's theorem that it is a direct product of matrix rings. Therefore,
it seems only natural to study if the Hamiltonian cycle can be lifted modulo the Jacobson radical.

\bigskip

\begin{Lemma}\label{thm:Ham3}
 If $\tau(R /J)$ is Hamiltonian, then $\tau(R)$ is Hamiltonian.
\end{Lemma}

\medskip
\begin{proof}
First, note that if $x \in Z(R)$ and $j\in J$, then $x+j \in Z(R)$, since otherwise $x+j=u\in R^*$ would imply that $x=u-j=u(1-u^{-1}j) \in R^* (1+J)\subseteq R^*$.

Let $x_1+J$ \, \textemdash\, $x_2+J$\, \textemdash \, \ldots \,  \textemdash\, $x_m+J$ be a Hamiltonian cycle in $R/J$. Note that $x_i+x_{i+1}+J \in Z(R/J)$ implies that $x_i+x_{i+1} \in Z(R)$ by Lemma \ref{lemma:R/J}
and thus $x_i+x_{i+1}+j+j' \in Z(R)$ for any $j,j' \in J$ by the above remark. If $J=\{j_1,j_2,\ldots,j_k\}$, it follows that 
$x_1+j_1$ \, \textemdash \, \ldots \, \textemdash \, \, $x_m+j_1$\, \textemdash \, $x_1+j_2$ \, \textemdash\, \ldots  \, \textemdash \, $x_m+j_2$\, \textemdash \,  \ldots  \,\textemdash \, $x_1+j_k$\, \textemdash \, \ldots \,  \textemdash\, $x_m+j_k$
is a Hamiltonian cycle in $\tau(R)$.
\end{proof}

\bigskip

We are now in the position to prove the main theorem of 
this section.

\bigskip

\begin{Theorem}
Let $R$ be a finite ring. Then, the graph $\tau(R)$ is Hamiltonian if and only if $R$ is not local.
\end{Theorem}

\begin{proof}
 If $R$ is local, then the graph $\tau(R)$ is not connected  by Proposition \ref{thm:local} and thus not Hamiltonian. 
 Otherwise, let  $R/J=M_{n_1}(F_1) \times \ldots \times M_{n_k}(F_k)$, with $k\geq 2$ or $n_1 \geq 2$. 
Then the theorem follows by Lemmas \ref{thm:Ham1}, \ref{thm:Ham2} and \ref{thm:Ham3}. 
\end{proof}


\bigskip
\bigskip
\section{The domination number}
\bigskip

The set of vertices $\Dd \subseteq V(G)$ in a graph $G$ is called a \emph{dominating set} if every vertex in $V(G) \backslash \Dd$ has
a neighbour in $\Dd$. The \emph{domination number $\gamma(G)$} is the minimum size of a dominating set in $G$. In this section, we find an upper bound for the domination number of an arbitrary finite ring.  Compare this result with the domination  number of a commutative finite ring in \cite[Theorem 4.1]{Shekarriz}.

\bigskip
As a consequence of Proposition \ref{thm:local} we obtain the following result, see also \cite[Theorem 4.1]{Shekarriz}.

\bigskip

\begin{Proposition}\label{thm:taulocal}
 If $R$ is a local ring, then $$\gamma(\tau(R))=\begin{cases}
 |R/J|, & \text{if } \ch(R)=2^k,\\
  \frac{1}{2}(|R/J|+1), & \text{otherwise}.
 \end{cases}$$
 \qed
\end{Proposition}

\bigskip

We now proceed to investigate the total graphs of arbitrary finite rings. We will again see that the full matrix ring over a field is of special importance,
so we shall first investigate the existence of a dominating set in the matrix setting.

\bigskip

\begin{Lemma}\label{thm:MnF}
 For every finite field $F$ and integer $n \geq 2$, the set $$\Dd=\{x E_{1j}; \; x \in F^*, \, 1 \leq j \leq n\} \cup \{0\}$$ 
 is a dominating set for $\tau(M_n(F))$.
\end{Lemma}

\medskip

\begin{proof}
  Choose an arbitrary matrix $A \in M_n(F)$ and denote by $A(1,j)$ the submatrix of  matrix $A$ without the first row and the $j$-th column.  
  If $A$ is a zero divisor, then there is an edge from $A$ to the zero matrix.
  If $A$ is invertible, then its determinant is nonzero, therefore there exists $j$,
  $1 \leq j \leq n$ such that the determinant of the $(n-1)\times (n-1)$ submatrix $A(1,j)$ is nonzero.  Since for any $x \in F$ we have
  $\det(A+xE_{1j})=\det(A)+x\det(A(1,j))$, we can choose $x=-\det(A(1,j))^{-1}\det(A)$ so that
  $\det(A+xE_{1j})=0$ and thus the matrix $A$ is connected to the matrix $xE_{1j}$.
\end{proof}

\bigskip

The next proposition states that the domination number does not change modulo the Jacobson radical.

\bigskip

\begin{Proposition}\label{thm:R/J}
 For every finite ring $R$ and its Jacobson radical $J=J(R)$ we have
  $$\gamma(\tau(R))=\gamma(\tau(R/J)).$$
\end{Proposition}

\medskip

\begin{proof}
Suppose $\Dd=\{a_1,\ldots,a_m\}$ is a dominating set for $\tau(R)$ and let $b \in R \backslash \Dd$. Since $\Dd$
is a dominating set, there exists  $i$, $1 \leq i \leq m$, such that $b+a_i \in Z(R)$.
By Lemma \ref{lemma:R/J} it follows that $(b+J)+(a_i+J) \in Z(R/J)$, so  $\{a_1+J,\ldots,a_m+J\}$ is a dominating set for 
$\tau(R/J)$, which implies $\gamma(\tau(R))\geq \gamma(\tau(R/J))$. 

Suppose now $\Dd=\{a_1+J,\ldots,a_m+J\}$ is a dominating set for $\tau(R/J)$ and let 
$b \in R/J \backslash \Dd$. There exists $i$, $1 \leq i \leq m$, such that 
$(b+J)+(a_i+J) \in Z(R/J)$ and by Lemma \ref{lemma:R/J} it follows that $b+a_i \in Z(R)$. So, $\{a_1,\ldots,a_m\}$ is a dominating set for
 $\tau(R)$ and  $\gamma(\tau(R))\leq \gamma(\tau(R/J))$.
\end{proof}

\bigskip

\begin{Lemma}\label{thm:RS}
For all finite rings $R$ and $S$ we have
  $$\gamma(\tau(R \times S))=\min\{\gamma(\tau(R)),\gamma(\tau(S))\}.$$
\end{Lemma}

\medskip
\begin{proof}
Suppose first $\{a_1,\ldots,a_m\}$ is a  dominating set for $\tau(R)$ and $\{b_1,\ldots,b_n\}$ is a  dominating set for $\tau(S)$.
Clearly, $\{(a_1,0),\ldots,(a_m,0)\}$ and $\{(b_1,0),\ldots,(b_m,0)\}$ are dominating sets for $R \times S$ and therefore
  $\gamma(\tau(R \times S))\leq \min\{\gamma(\tau(R)),\gamma(\tau(S))\}$.
  
 Suppose now $\{(a_1,b_1),\ldots,(a_m,b_m)\}$ is a  dominating set for $\tau(R \times S)$. If $\{a_1,\ldots,a_m\}$ is a  dominating set
  for $\tau(R)$, then   $\min\{\gamma(\tau(R)),\gamma(\tau(S))\} \leq \gamma(\tau(R)) \leq \gamma(\tau(R \times S))$. Otherwise, 
  there exists $x\in R$, such that $x+a_i\notin Z(R)$ for $i=1,\ldots,m$. Choose an arbitrary $y \in S$. There exists $i$, such that
  $(x,y)+(a_i,b_i) \in Z(R \times S)$, i.e. there exists a nonzero $(z,w) \in R \times S$, such that $((x,y)+(a_i,b_i))(z,w)=(0,0)$. This implies that
  $(x+a_i)z=0$ and $(y+b_i)w=0$. Since $x+a_i\notin Z(R)$, $z=0$ and thus $w \ne 0$. This implies that $y+b_i\in Z(S)$ and therefore
  $\{b_1,\ldots,b_m\}$ is a  dominating set  for $\tau(S)$. It follows that
  $\min\{\gamma(\tau(R)),\gamma(\tau(S))\} \leq \gamma(\tau(R \times S))$. 
\end{proof}

\bigskip

Now, we can apply Wedderburn's theorem, 
Lemma \ref{thm:MnF}, Proposition \ref{thm:R/J} and Lemma \ref{thm:RS} to obtain the following result.

\bigskip

\begin{Theorem}\label{thm:dom}
 If $R$ is an arbitrary finite ring and $R/J=M_{n_1}(F_1) \times \ldots \times M_{n_k}(F_k) $, then
  $$\gamma(\tau(R)) \leq \min_i \{n_i (|F_i|-1)+1\}.
$$
    \hfill$\blacksquare$
\end{Theorem}

\bigskip

Obviously, if $n_i=1$ for some $i$, we can find a ring with a smaller dominating number than the bound in Theorem \ref{thm:dom}. For example, $\gamma(\tau(\ZZ_3))=2<3$. 
However, a straightforward calculation shows that $\gamma(\tau(M_2(\ZZ_2))) = 3=2 (|\ZZ_2|-1)+1$ and a calculation by computer shows that 
$\gamma(\tau(M_3(\ZZ_2))) = 4=3 (|\ZZ_2|-1)+1$.  We have a reason to believe the following.

\bigskip

\begin{Conjecture}
 If $R$ is arbitrary ring and $R/J=M_{n_1}(F_1) \times \ldots \times M_{n_k}(F_k) $, where 
 $n_i\geq 2$ for all $i$, then
  $$\gamma(\tau(R)) = \min_i \{n_i (|F_i|-1)+1\}.$$
\end{Conjecture}

%
%
%
%
%

\end{document}